\documentclass{elsarticle}
\usepackage{amssymb, latexsym, comment, amsmath, amsthm, upgreek, epsf, epsfig, color, natbib, graphicx}

\makeatletter
\def\ps@pprintTitle{%
\let\@oddhead\@empty
\let\@evenhead\@empty
\def\@oddfoot{\centerline{\thepage}}%
\let\@evenfoot\@oddfoot}
\makeatother

\newtheorem{thm}{Theorem}[section]
\newtheorem{lemma}[thm]{Lemma}
\newtheorem{prop}[thm]{Proposition}
\newtheorem{cor}[thm]{Corollary}

\newtheorem{ex}[thm]{Example}

\begin{document}
\begin{frontmatter}
\title{Characterizing Round Spheres Using Half-Geodesics}
\author[1]{Ian M Adelstein} 
\author{Benjamin Schmidt \fnref{fn2}}
\begin{abstract} A half-geodesic is a closed geodesic realizing the distance between any pair of its points. All geodesics in a round sphere are half-geodesics.  Conversely, this note establishes that Riemannian spheres with all geodesics closed and \textit{sufficiently} many half-geodesics are round.
\end{abstract}
\fntext[fn1]{ian.adelstein@yale.edu}
\fntext[fn2]{schmidt@math.msu.edu}
\end{frontmatter}

%
%
%
%

\section{Introduction}

Geodesics, locally distance minimizing curves, need not globally minimize distance. For example, closed geodesics do not minimize beyond half their prime length. \emph{Half-geodesics}, defined as closed geodesics that restrict to minimizing geodesics on \textit{all} subsegments of half the prime length, are the maximally minimizing closed geodesics.  

A closed Riemannian manifold with non-zero fundamental group has a half-geodesic; a non-contractible closed curve of shortest length is one such \cite[Lemma 4.1]{Sor}. Contrastingly, closed simply connected Riemannian manifolds have a closed geodesic \cite{LF} but need not have a half-geodesic as discussed below.  \emph{Blaschke manifolds}, defined as closed Riemannian manifolds with equal injectivity radius and diameter, lie at the other extreme.

\begin{thm}\label{one} A closed Riemannian manifold is Blaschke if and only if all geodesics are (closed) half-geodesics. 
\end{thm}

  A \textit{Besse metric} on a smooth manifold is a Riemannian metric with all geodesics closed. Spheres in each dimension admit Besse metrics that are not round \cite[Chapter 4]{besse}.

\begin{thm}\label{two} A Besse $n$-sphere $M$ is Blaschke if
\begin{enumerate}
\item all prime geodesics have equal length, and
\item each point in $M$ lies in a half-geodesic.
\end{enumerate}
\end{thm}

Hypothesis 1 is satisfied by all Besse metrics on the $n$-dimensional sphere, except possibly when $n=3$ \cite{GG,RW}.  Moreover, a Riemannian metric on a sphere is Blaschke if and only if it is round  \cite[Appendix D]{besse}. These facts and Theorem \ref{two} imply the next Corollary.
\begin{cor}\label{c1}
If $n \neq 3$, then a Besse metric on the $n$-sphere in which all points lie in a half-geodesic is round.
\end{cor}

In Theorem \ref{two}, the union of half-geodesics contains $M$.  Alternatively, one may assume this union contains a suitable codimension one totally-geodesic submanifold.

\begin{thm}\label{three} A Besse n-sphere $M$ is Blaschke if 
\begin{enumerate}
\item all prime geodesics have equal length, and
\item $M$ has a codimension one totally-geodesic submanifold $N$ each point of which lies in a half-geodesic of $M$ tangent to $N$. 
\end{enumerate}
\end{thm}

\begin{cor}\label{c2}
A Besse 2-sphere has a half-geodesic if and only if it is round.
\end{cor}

\section{Contextualizing Examples}\label{examples}

As Blaschke spheres are round, Theorem \ref{one} implies that Riemannian spheres with all geodesics (closed) half-geodesics are round. The following examples describe non-round spheres with zero, any natural number, or a continuum of half-geodesics.

\begin{ex}\label{none}\normalfont Metrics on the 2-sphere with no half-geodesics abound. Metrics close to singular spaces appear in \cite{Ade, Ho}. Besse metrics with a common prime length exceeding twice the diameter appear in \cite{Bal}. Corollary~\ref{c2} supplements these examples.  \end{ex}

\begin{ex}\label{ellipsoid2}\normalfont There exist Riemannian 2-spheres arbitrarily close to a round sphere with one half-geodesic: By \cite[Chapter IX, Theorem 4.1]{Morse} there exists $\epsilon > 0$ such that if $$1<a<b<c<1+\epsilon,$$ then closed geodesics in the triaxial ellipsoid $$(x/a)^2+(y/b)^2+(z/c)^2=1$$ have length larger than 1000 except for the three coordinate plane sections. The coordinate sections $x=0$ and $y=0$ contain the largest axis and fail to minimize between the antipodal points on their second axis.  The remaining coordinate section $z=0$ is the only half-geodesic. 
\end{ex}

\begin{ex} \label{exact}\normalfont For each integer $n \geq 2$, there exist Riemannian metrics on the 2-sphere with exactly $n$ half-geodesics \cite{Ade}; a sequence of such metrics is constructed converging to a doubled regular $2n$-gon in the Gromov-Hausdorff metric.  
\end{ex}

\begin{ex}\label{oblate}\normalfont There exist non-round Riemannian 2-spheres arbitrarily close to a round sphere swept out by a continuum of half-geodesics: Let $0<c<1$ and consider the oblate ellipsoid $$x^2 + y^2 + (z/c)^2 = 1.$$ This ellipsoid is rotationally symmetric about the $z$-axis. The meridians are half-geodesics of length twice the diameter. The equator is a closed geodesic with length greater than twice the diameter, and therefore is not a half-geodesic. 
\end{ex}

\section{Proofs}

\begin{proof}[Proof of Theorem~\ref{one}] Geodesics in a Blaschke manifold $M$ are (closed) half-geodesics by \cite[Proposition 5.39]{besse}.

Conversely, assume all geodesics in $M$ are (closed) half-geodesics.  Let $SM$ denote the unit-sphere bundle of $M$ and $$C \colon SM \to \mathbb{R}$$ the cut time function.  Then $C(v)$ equals half the prime length of the closed geodesic with initial velocity $v \in SM$.  The values of $C$ are discrete by \cite{wa}.  As $C$ is continuous and $SM$ is connected, $C$ is constant.  Therefore, all prime closed geodesics have equal length $2C$ and minimize up to but not beyond length $C$.  Conclude $$inj\leq diam \leq C \leq inj.$$
\end{proof}

A proof of the next Proposition is presented at the end of this section. 

\begin{prop}\label{key}
Let $M$ a Besse $n$-sphere with diameter one in which all prime geodesics have length two. If $p \in M$ is a diameter realizing point, then the restriction of $\exp_p$ to the unit sphere in $T_pM$ is a point map. 
\end{prop}

\begin{proof}[Proof of Theorem~\ref{two} assuming Proposition \ref{key}] 
Rescale the metric, if necessary, so prime geodesics have length two.  Then $inj\leq diam \leq 1$.  It suffices to prove $1\leq inj$, or equivalently, each geodesic in $M$ of length one is minimizing.  

Fix $p \in M$ and a unit-length vector $v \in T_pM$.  Let $$\gamma_v:[0,1] \rightarrow M$$ denote the length one geodesic segment with $\dot{\gamma}_v(0)=v$.  By assumption, the point $p$ lies in a half-geodesic of length two.  Let $q$ denote the point in this half-geodesic with $d(p,q)=1$.  By Proposition \ref{key}, $\gamma_v(1)=q$.  As the endpoints of $\gamma_v$ are distance one apart and $\gamma_v$ has length one, $\gamma_v$ is minimizing.
\end{proof}

\begin{proof}[Proof of Theorem~\ref{three} assuming Theorem \ref{two} and Propostion \ref{key}] 
Rescale the metric, if necessary, so prime geodesics have length two.  Fix $p \in N$ and a unit vector $v \in T_pM$ that is \textit{not} tangent to $N$ at $p$. Let $$\gamma_v:[0,1] \rightarrow M$$ denote the length one geodesic determined by $\dot{\gamma}_v(0)=v$.  By \cite{bangert1983}, it suffices to prove  $$t \in (0,1) \implies \gamma_v(t) \notin N\,\,\,\,\,\, \text{and}\,\,\,\,\,\, \gamma_v(1) \in N.$$ 

A half-geodesic in $M$ tangent to $N$ is also a half-geodesic in $N$.  By Theorem \ref{two}, $N$ is a round sphere of diameter one.  Let $q$ be the point antipodal to $p$ in $N$.  By Proposition \ref{key}, $$\gamma_v(1)=q \in N.$$  Moreover, as in the previous proof, $\gamma_v:[0,1] \rightarrow M$ is a minimizing geodesic.  As $N$ is round, each geodesic leaving $p$ starting tangent to $N$ is minimizing up to length one.  As two minimizing geodesic segments with a common initial point can only intersect again in their endpoints, the requisite condition $$ t\in (0,1) \implies \gamma_v(t) \notin N$$ is satisfied.
\end{proof}

Some well-known preliminary material is needed in the proof of Proposition \ref{key}:  

Given $(x,y) \in M \times M$, let $\Lambda(x,y)$ denote the pathspace consisting of piecewise smooth maps $c:[0,1]\rightarrow M$ with $c(0)=x$ and $c(1)=y$.  Define the length and energy functions $$L: \Lambda(x,y) \rightarrow [0,\infty)\,\,\,\,\,\text{and}\,\,\,\,\,E:\Lambda(x,y) \rightarrow [0,\infty)$$  by $$L(c)=\int_0^1 ||\dot{c}(t)||dt\,\,\,\,\, \text{and}\,\,\,\,\, E(c)=\int_0^1 ||\dot{c}(t)||^2dt.$$ The Cauchy-Schwartz inequality implies $$L^2(c) \leq E(c)$$ with equality holding if and only if $c$ has constant speed $||\dot{c}(t)||$.  Given $e \in [0,\infty)$, let $$\Lambda_e(x,y)=E^{-1}([0,e))\,\,\,\,\, \text{and}\,\,\,\, \Lambda^e(x,y)=E^{-1}([0,e]).$$

A path is a critical point for $E$ if and only if the path is a geodesic.  The index of a critical point $\gamma \in \Lambda(x,y)$ equals the number of parameters $s \in (0,1)$ for which $\gamma(s)$ is conjugate to $\gamma(0)$ along $\gamma$, counted with multiplicities.

\begin{lemma}\label{index}
If $M$ is a Besse n-sphere in which all prime geodesics have length two, then a geodesic segment of length greater than two has index at least n.  
\end{lemma}

\begin{proof}
After demonstrating some pair of points in $M$ are conjugate along a geodesic of length less than two, the proof of \cite[Lemma 2.4]{Sch} applies verbatim.

As prime geodesics have length two, the diameter does not exceed one. By Klingenberg's estimate, the injectivity radius equals the minimum of the conjugate radius and one, half the length of a shortest closed geodesic.  If the conjugate radius achieves this minimum, there is nothing left to prove.  If one achieves this minimum, then $M$ is a Blaschke sphere, hence round, and therefore has conjugate radius one.
\end{proof}

Given a continuous map $h:S^{n-1} \rightarrow \Lambda(x,y),$ define the \textit{associated map} $$\hat{h}:S^{n-1} \times [0,1] \rightarrow M$$ by $\hat{h}(\theta,t)=h(\theta)(t)$ for each $(\theta,t) \in S^{n-1} \times [0,1]$. If $h$ represents a non-zero class in $\pi_{n-1}(\Lambda(x,y))\cong \mathbb{Z}$, then the associated map $\hat{h}$ is surjective.

Given $\epsilon<1$, $c \in \Lambda(x,y)$, a point $z\in M$ satisfying $d(y,z)=\epsilon$, and a unit-speed geodesic $\tau:[0,\epsilon]\rightarrow M$ with $\tau(0)=y$ and $\tau(\epsilon)=z$, define $$\tau* c\in \Lambda(x,z)$$ by  \[ 
\tau* c(t)=
\begin{cases}
c(\frac{t}{1-\epsilon}) & \text{for } t\in [0,1-\epsilon]\\
\tau(t-(1-\epsilon)) & \text{for } t\in[1-\epsilon,1].\\
\end{cases}
\]  Then $$E(\tau*c)=\frac{E(c)}{1-\epsilon}+\epsilon.$$ Given a map $f:S^{n-1} \rightarrow \Lambda(x,y)$, the map $\tau f:S^{n-1}\rightarrow \Lambda(x,z)$ defined by $\tau f(\theta)=\tau*f(\theta)$ for each $\theta \in S^{n-1}$, represents a nonzero homotopy class of maps if and only if $f$ represents a nonzero homotopy class of maps.

\begin{proof}[Proof of Proposition \ref{key}]
The proof follows that of \cite[Theorem 1]{Sch}. Some details are ommitted. 

Let $p \in M$ be a diameter realizing point and $d(p,q)=1$. Since prime geodesics have length two, the critical values of $E:\Lambda(p,q) \rightarrow [0,\infty)$ are the squares of odd integers (cf. \cite[Lemma 2.2]{Sch}).

Fix a positive $\epsilon$ smaller than the injectivity radius at $q$.  In particular, $\epsilon<1$.  Let $S$ denote the sphere with center $q$ and radius $\epsilon$ and let $$D=\max\{d(p,s)\,\vert\, s \in S\}.$$  The strict triangle inequality implies that for each $s \in S$, $d(p,s)\geq 1-\epsilon$ with equality if and only if $s$ lies along a minimizing geodesic joining $p$ to $q$.  Therefore, $1-\epsilon \leq D$ with equality implying the Proposition.  \vskip 5pt

\noindent \textbf{Claim} If there is continuous map $g:S^{n-1} \rightarrow \Lambda_9(p,q)$ representing a nonzero class in $\pi_{n-1}(\Lambda(p,q))$, then the equality $1-\epsilon=D$ holds.\vskip 5pt

\begin{proof}[Proof of Claim] 
Seeking a contradiction, assume that $1-\epsilon<D$.  The strict triangle inequality implies that for each $s \in S$, $d(p,s)<d(p,q)+d(q,s)=1+\epsilon$.  Conclude $D<1+\epsilon$ and $$1<(D+\epsilon)^2<(1+2\epsilon)^2<9.$$ Therefore $E$ has no critical points in $E^{-1}([(D+\epsilon)^2,9))$ and $g$ is homotopic to a map $h:S^{n-1} \rightarrow \Lambda_{(D+\epsilon)^2}(p,q)$.  The map $h$, like $g$, represents a nonzero class in $\pi_{n-1}(\Lambda(p,q))$.  The associated map $\hat{h}:S^{n-1} \times [0,1] \rightarrow M$ is therefore surjective.

Choose $s_0 \in S$ with $d(p,s_0)=D$. There exists $(\theta_0,t_0) \in S^{n-1} \times [0,1]$ with $\hat{h}(\theta_0)(t_0)=s_0$. The curve $$h(\theta_0) \in \Lambda(p,q)$$ satisfies $$L(h(\theta_0))\leq \sqrt{E(g(\theta_0))}<D+\epsilon.$$ As $d(s_0,q)=\epsilon$, the restriction of $h(\theta_0)$ to $[t_0,1]$ has length at least $\epsilon$.  Therefore, the restriction of $h(\theta_0)$ to $[0,t_0]$ has length less than $D$, contradicting $d(p, s_0)=D$.  
\end{proof}

It remains to establish the existence of a continuous map $g:S^{n-1} \rightarrow \Lambda_9(p,q)$ representing a nonzero class in $\pi_{n-1}(\Lambda(p,q))$. To this end, choose $z \in M$ such that $p$ is not conjugate to $z$ along any geodesic in $M$ and such that $\epsilon:=d(q,z)$ is less than the minimum of the injectivity radius of $q$ and $1/2$.  Then critical points of $$E: \Lambda(p,z) \rightarrow \mathbb{R}$$ are non-degenerate.

Let $\bar{f}: S^{n-1} \rightarrow \Lambda(p,z)$ be a continuous map representing a nonzero class in $\pi_{n-1}(\Lambda(p,z))$.  By Lemma \ref{index}, critical points of length greater than two have index at least $n$.  It follows (cf. \cite[Theorem 2.5.16]{klingbook}) $\bar{f}$ is homotopic to a map $$f:S^{n-1} \rightarrow \Lambda^{4.1}(p,z).$$  Let $\tau:[0,\epsilon] \rightarrow M$ be a minimizing geodesic with $\tau(0)=z$ and $\tau(\epsilon)=q$ and define $$g:S^{n-1} \rightarrow \Lambda(p,q)$$ by $g=\tau f.$ For each $\theta \in S^{n-1}$, $$E(g(\theta))=\frac{E(f(\theta))}{1-\epsilon}+\epsilon\leq 2E(f(\theta))+\frac{1}{2}\leq 8.7$$ concluding the proof.
\end{proof}

\pagebreak


\bibliographystyle{abbrvnat}
\bibliography{polygon}

\end{document}